\documentclass[english]{amsart}

\usepackage{esint}
\usepackage[svgnames]{xcolor} 
\usepackage[colorlinks,citecolor=red,pagebackref,hypertexnames=false,breaklinks]{hyperref}
\usepackage{pgf,tikz}
\usepackage{pdfsync}

\usepackage{dsfont}
\usepackage{url}
\usepackage[utf8]{inputenc}
\usepackage[T1]{fontenc}
\usepackage{lmodern}
\usepackage{babel}
\usepackage{mathtools}  
\usepackage{amssymb}
\usepackage{lipsum}
\usepackage{mathrsfs}
\usepackage{color}

\newtheorem{theorem}{Theorem}[section]

\newtheorem{lemma}{Lemma}[section]

\newtheorem{corollary}{Corollary}[section]

\numberwithin{equation}{section}

\title[Borg Levinson type theorem]{A simple proof of a multidimensional Borg-Levinson type theorem}

\author[Choulli]{Mourad Choulli}
\address{Universit\'e de Lorraine, 34 cours L\'eopold, 54052 Nancy cedex, France}
\thanks{The author supported by the grant ANR-17-CE40-0029 of the French National Research Agency ANR (project MultiOnde).}
\email{mourad.choulli@univ-lorraine.fr}

\date{\today}


\begin{document}

\begin{abstract}
We provide a simple and short proof of a multidimensional Borg-Levinson type theorem. Precisely, we prove that the knowledge of spectral boundary data determine uniquely the corresponding potential  appearing in the Sch\"odinger operator on an admissible Riemannian manifold. We also sketch the proof of uniqueness in the case of incomplete spectral boundary data. The new results in the present work complete those existing in the literature.
\end{abstract}

\maketitle

\section{Introduction}\label{section1}

Let $\mathcal{M}=(\mathcal{M},\mathfrak{g})$ be a smooth compact Riemannian manifold, of dimension $n\ge3$,  with boundary $\partial \mathcal{M}$. We recall that, in local coordinates,  the usual Laplace-Beltrami operator is given by
\[
\Delta_\mathfrak{g}=\frac{1}{\sqrt{|\mathfrak{g}|}}\sum_{j,k=1}^n\frac{\partial}{\partial x_j}\left(\sqrt{|\mathfrak{g}|}\mathfrak{g}^{jk}\frac{\partial}{\partial x_k}\, \cdot \right).
\]

We define, for $q\in L^\infty (\mathcal{M},\mathbb{R})$,  the unbounded operator $L=L(q)$ acting on $L^2(\mathcal{M})$ as follows
\[
L=-\Delta _{\mathfrak{g}}+q\quad \mbox{with}\quad D\left(L \right)=H_0^1(\mathcal{M})\cap H^2(\mathcal{M}).
\]
As $L$ is self-adjoint operator with compact resolvent, its spectrum is reduced to a sequence of eigenvalues:
\[
-\infty < \lambda_1\le \ldots \lambda_k\le \ldots \quad \mbox{and}\quad \lambda_k\rightarrow \infty \; \mbox{as}\; k\rightarrow \infty.
\]
Of course $\lambda_k=\lambda_k (q)$, $k\ge 1$.

Furthermore, there exists $(\phi_k)$, $\phi_k=\phi_k(q)$, $k\ge 1$, an orthonormal basis of $L^2(\mathcal{M})$ consisting of eigenfunctions. Each $\phi_k$ is associated to $\lambda_k$, $k\ge 1$.

For simplicity convenience, we use the notation 
\[
\psi_k=\partial_\nu  \phi_k|{_{\partial \mathcal{M}}}, \quad k\ge 1,
\]
where $\nu$ is the unit normal exterior vector field on $\partial \mathcal{M}$ with respect to $\mathfrak{g}$.

When $q\in L_+^\infty (\mathcal{M})=\{q\in L^\infty (\mathcal{M},\mathbb{R});\; q\ge 0\}$, $0$ is not in the spectrum of $L$ and all the eigenvalues of $L$ are positive. Whence, for any $f\in H^{1/2}(\partial \mathcal{M})$, the BVP
\begin{equation}\label{bvp}
(-\Delta _{\mathfrak{g}}+q)u=0\; \mbox{in}\; \mathcal{M}\quad \mbox{and}\quad \mbox u=f\; \mbox{on}\; \partial \mathcal{M}
\end{equation}
admits a unique solution $u=u(q,f)\in H^1(\mathcal{M})$ with $\partial_\nu u\in H^{-1/2}(\partial \mathcal{M})$. Moreover 
\[
\|\partial_\nu u\|_{H^{-1/2}(\partial \mathcal{M})}\le C\|f\|_{H^{1/2}(\partial \mathcal{M})},
\]
where the constant $C$ is independent of $f$. In other words, the mapping
\[
\Lambda(q) :f\in H^{1/2}(\partial \mathcal{M})\rightarrow \partial_\nu u\in H^{-1/2}(\partial \mathcal{M})
\]
defines a bounded operator. The operator $\Lambda (q)$ is usually called the Dirichlet-to-Neumann map associated to $q$.

Pick $\tilde{q}\in L_+^\infty (\mathcal{M})$. Set $\tilde{\lambda}_k=\lambda_k (\tilde{q})$, $\tilde{\phi}_k=\phi_k (\tilde{q})$ and $\tilde{\psi}_k=\psi_k (\tilde{q})$. Introduce then the notation 
\[
\mathfrak{D}(q,\tilde{q})=\sum_{k\ge 1}\left|\lambda_k-\tilde{\lambda}_k\right|+\sum_{k\ge 1} \|\tilde{\psi}_k-\psi_k\|_{L^2(\partial \mathcal{M})}.
\]

We aim to prove the following result:
 
 \begin{theorem}\label{theorem1} 
 Let $\aleph >0$. Suppose that $0\le q\le \aleph$, $0\le \tilde{q}\le \aleph$ and $\mathfrak{D}(q,\tilde{q})<\infty$. Then $\Lambda (q) -\Lambda(\tilde{q})$ extends to a bounded operator on $L^2(\partial \mathcal{M})$ and
\[
\|\Lambda (q) -\Lambda(\tilde{q})\|_{\mathscr{B}(L^2(\partial \mathcal{M}))}\le C\mathfrak{D}(q,\tilde{q}),
\]
where the constant $C$ only depends on $n$, $\mathcal{M}$ and $\aleph$.
 \end{theorem}

Let $(\mathscr{M},g)$ a compact Riemannian manifold with boundary $\partial \mathscr{M}$. We say that $\mathscr{M}$ is admissible if $\mathscr{M}\Subset \mathbb{R}\times \mathscr{M}_0$, for some (n-1)-dimensional simple manifold $(\mathscr{M}_0,g_0)$,  and if $g=c(\mathfrak{e}\oplus g_0)$, where $\mathfrak{e}$ is the Euclidean metric on $\mathbb{R}$ and $c$ is a smooth positive function on $\mathscr{M}$. A compact Riemannian manifold with boundary $(\mathscr{M}_0,g_0)$ is called simple  if, for any $x\in \mathscr{M}_0,$ the exponential map $\exp_x$ with its maximal domain of definition is a diffeomorphism onto $\mathscr{M}_0$, and if $\partial \mathscr{M}_0$ is strictly convex (which means that the second fundamental form of $\partial \mathscr{M}_0\hookrightarrow \mathscr{M}_0$ is positive definite).

Observing that 
\[
\|\Lambda (q) -\Lambda(\tilde{q})\|_{\mathscr{B}(H^{1/2}(\partial \mathcal{M}),H^{-1/2}(\partial \mathcal{M}))}\le \|\Lambda (q) -\Lambda(\tilde{q})\|_{\mathscr{B}(L^2(\partial \mathcal{M}))},
\]
and taking into account that $\mathfrak{D}(q,\tilde{q})=\mathfrak{D}(q-\lambda ,\tilde{q}-\lambda)$, for any $\lambda$ not belonging to the spectrum of $L(q)$ and $L(\tilde{q})$, we deduce as a straightforward consequence of Theorem \ref{theorem1} and \cite[Theorem 1]{CS} the following stability inequality: 

\begin{corollary}\label{corollary1}
Let $\aleph>0$, $t\in (0,1/2)$ and assume that $\mathcal{M}$ is admissible with $\mathcal{M}\Subset \mathbb{R}\times \mathcal{M}_0$. Then there exist two constants $C>0$ and $0<c<1$, only depending on $n$, $\mathcal{M}$ and $\aleph$, so that, for any $q,\tilde{q}\in L^\infty (\mathcal{M},\mathbb{R}) \cap H^t(\mathcal{M})$ satisfying 
\[ 
\|q\|_{L^\infty(\mathcal{M})\cap H^t(\mathcal{M})}\le \aleph ,\quad \|\tilde{q}\|_{L^\infty(\mathcal{M})\cap H^t(\mathcal{M})}\le \aleph\quad  \mbox{and}\quad  \mathfrak{D}(q,\tilde{q})\le c, 
\]
we have
\[
\| q-\tilde{q}\|_{L^2(\mathbb{R},H^{-3}(\mathcal{M}_0))}\le C\left| \ln \left[\mathfrak{D}(q,\tilde{q})+\left|\ln \left( \mathfrak{D}(q,\tilde{q})\right)\right|^{-1}\right]\right|^{-t/4}.
\]
\end{corollary}

This corollary says  in particular that in an admissible manifold the spectral boundary data $(\lambda_k(q),\psi_k(q))_{k\ge 1}$ determine uniquely $q\in L^\infty(\mathcal{M},\mathbb{R})\cap H^t(\mathcal{M})$:

\begin{theorem}\label{theorem2}
Let $t\in (0,1/2)$ and assume that $\mathcal{M}$ is admissible. If $q,\tilde{q}\in L^\infty(\mathcal{M},\mathbb{R})\cap H^t(\mathcal{M})$ satisfy
\[
\lambda_k(q)=\lambda_k(\tilde{q})\quad \mbox{and}\quad \psi_k(q)=\psi_k(\tilde{q}),\quad k\ge 1,
\]
then $q=\tilde{q}$.
\end{theorem}

Results of the same kind as in Theorem \ref{theorem2} are known in the literature as multidimensional Borg-Levinson type theorems.

The paper by Nachman, Sylvester and Uhlmann \cite{NSU} was the starting point of new developments and results on  multidimensional Borg-Levinson type theorems. We quote here the following short list of references: \cite{AS, Bel87, Bel92, BCDKS, BCY,BCY2,Ch,ChS,Is,KK,KKL,KKS,Ki1,KOM}. Of course many other works on spectral inverse problems exist in the literature. A  short survey on multidimensional Borg-Levinson type theorems can be found in \cite{BCDKS}. The case of unbounded potentials was considered by Pohjola in  \cite{Po}. We also mention a new stability inequality in two dimensional case, recently  established by Imanuvilov and Yamamoto in \cite{IY}.

\section{Proof of the main result}

Henceforward, the usual scalar products on $L^2(\mathcal{M})$ and $L^2(\partial \mathcal{M})$ are denoted respectively by $(\cdot |\cdot )$ and $\langle \cdot |\cdot \rangle$.

The following lemma will be useful in the sequel.
\begin{lemma}\label{lemma0}
Let $q\in L^\infty_+(\mathcal{M})$, $f\in H^{1/2}(\mathcal{M})$ and $u=u(q,f)$. We have
\begin{equation}\label{eq0}
u=-\sum_{k\ge 1}\frac{\langle f|\psi_k\rangle}{\lambda_k}\phi_k
\end{equation}
and hence
\[
\sum_{k\ge 1}\frac{|\langle f|\psi_k\rangle|^2}{\lambda_k^2}=\|u\|_{L^2(\mathcal{M})}^2.
\]
\end{lemma}

\begin{proof}
As $(\phi_k)$ is an orthonormal basis of $L^2(\mathcal{M})$, we obtain
\begin{equation}\label{eq2}
u=\sum_{k\ge 1}(u|\phi_k)\phi_k.
\end{equation}
Using that $u$ is the solution of the BVP \eqref{bvp} and Green's formula in order to get
\begin{align*}
(u|\lambda_k\phi_k)&=(u| (-\Delta_{\mathfrak{g}}+q)\phi_k)
\\
&= ((-\Delta_{\mathfrak{g}}+q)u|\phi_k)-\langle u|\partial_\nu \phi_k\rangle
\\
&=-\langle f|\partial_\nu \phi_k\rangle .
\end{align*}
That is we have
\[
(u|\phi_k)=-\frac{\langle f|\psi_k\rangle}{\lambda_k}.
\]
This identity in \eqref{eq2} yields \eqref{eq0}.
\end{proof}

Pick $(\rho_k)$ an orthonormal basis of $L^2(\mathcal{M} )$, $s\ge 0$ and define 
\[
\mathscr{H}^s=\left\{w\in H^2(\mathcal{M} ),\; \sum_{k\ge 1}k^{4s/n}\left(k^{4/n}|(w|\rho_k)|^2+|(\Delta_\mathfrak{g}w|\rho_k)|^2\right)<\infty\right\}.
\]
We have by Parseval's identity
\begin{align*}
&\sum_{k\ge 1}k^{4s/n}\left(k^{4/n}|(w|\rho_k)|^2+|(\Delta_\mathfrak{g}w|\rho_k)|^2\right)
\\
&\qquad =\left\| \sum_{k\ge 1} k^{2(s+1)/n}(w|\rho_k)\rho_k\right\|_{L^2(\mathcal{M})}^2
+\left\| \sum_{k\ge 1} k^{2s/n}(\Delta_{\mathfrak{g}}w|\rho_k)\rho_k\right\|_{L^2(\mathcal{M})}^2.
\end{align*}
In consequence, the left hand side of this identity does not depend on the choice of the orthonormal basis $(\rho_k)$. We endow $\mathscr{H}^s$ with the norm
\[
\|w\|_{\mathscr{H}^s}=\|w\|_{H^2(\mathcal{M})}+\left(\sum_{k\ge 1}k^{4s/n}\left(k^{4/n}|(w|\rho_k)|^2+|(\Delta_\mathfrak{g}w|\rho_k)|^2\right)\right)^{1/2}.
\]

It is not hard check that $\mathscr{H}^s$ contains all the eigenfunctions $(\varphi_k)$ of the operator $-\Delta_{\mathfrak{g}}$, under Neumann boundary condition, and therefore it contains also the closure, with respect to $\|\cdot \|_{\mathscr{H}^s}$, of the vector space spanned by these eigenfunctions. This can be easily seen just  by taking $\rho_k=\varphi_k$, $k\ge 1$, in the definition of $\mathscr{H}^s$.

We associate to $\mathscr{H}^s$ the  trace space
\[
\mathfrak{H}^s=\{ g\in H^{3/2}(\partial \mathcal{M});\; g=w|_{\partial \mathcal{M}}\;  \mbox{for some}\; w\in \mathscr{H}^s\}
\]
that we equip  with its natural quotient norm
\[
\|g\|_{\mathfrak{H}^s}=\inf\left\{\|w\|_{\mathscr{H}^s};\; w|_{\partial \mathcal{M}}=g\right\}.
\]

\begin{lemma}\label{lemma1}
We fix $s>n/4$. Then, for any $g\in \mathfrak{H}^s$, the series $\sum_{k\ge 1}\frac{\langle g|\psi_k\rangle}{\lambda_k}\phi_k $ converges in $H^2(\mathcal{M})$ with
\begin{equation}\label{1}
\sum_{k\ge 1}\frac{|\langle g|\psi_k\rangle|}{\lambda_k}\|\phi_k \|_{H^2(\mathcal{M})}\le C\|g\|_{\mathfrak{H}^s},
\end{equation}
where the constant $C$ is independent of $g$.
\end{lemma}

\begin{proof}
Let $g\in \mathfrak{H}^s$ and $w\in \mathscr{H}^s$ arbitrary so that $w|_{\partial \mathcal{M}}=g$. From Green's Formula, we have
\begin{align*}
\langle g|\psi_k\rangle =\langle g|\partial _\nu \phi_k\rangle &=(w|\Delta_\mathfrak{g}\phi_k)-(\Delta_\mathfrak{g}w|\phi_k).
\\
&= (w|(-\lambda_k+q)\phi_k )-(\Delta_\mathfrak{g}w|\phi_k ).
\end{align*}
Therefore, taking into account that $\lambda_k\sim k^{2/n}$ as $k\rightarrow +\infty$ (by Weyl's asymptotic formula), we get
\[
|\langle g|\psi_k\rangle|\le C\left( k^{2/n}|(\phi_k|w)|+|(\phi_k|\Delta_\mathfrak{g}w)| \right),\quad k\ge 1,
\]
where the constant $C$ is independent of $w$. 

But $\|\phi_k\|_{H^2(\mathcal{M})}\le C\lambda_k$ for any $k\ge 1$. Hence
\[
\frac{|\langle g|\psi_k\rangle|}{\lambda_k}\|\phi \|_{H^2(\mathcal{M})}\le Ck^{-2s/n}\left[k^{2s/n}\left( k^{2/n}|(\phi_k|w)|+|(\phi_k|\Delta_\mathfrak{g}w)| \right)\right],\quad k\ge 1.
\]
Whence, the series $\sum_{k\ge 1}\frac{\langle g|\psi_k\rangle}{\lambda_k}\phi_k $ converges in $H^2(\mathcal{M})$ and using Cauchy-Schwarz's inequality  we find
\[
\sum_{k\ge 1}\frac{|\langle g|\psi_k\rangle|}{\lambda_k}\|\phi_k \|_{H^2(\mathcal{M})}\le  C\|w\|_{\mathscr{H}^s}.
\]
As $w\in \mathscr{H}^s$ is chosen arbitrary so that $w|_{\partial \mathcal{M}}=g$, inequality \eqref{1} follows.
\end{proof}

\begin{proof}[Proof of Theorem \ref{theorem1}]
In this proof $C$ is a generic constant depending only on $n$, $\mathcal{M}$ and $\aleph$.

In light of Lemma \ref{lemma0} and Lemma \ref{lemma1}, we have
\[
\Lambda(q)(f)= -\sum_{k\ge 1}\frac{\langle f|\psi_k\rangle}{\lambda_k}\psi_k,\quad f\in \mathfrak{H}^s
\]
and
\[
\Lambda(\tilde{q})(f)= -\sum_{k\ge 1}\frac{\langle f|\tilde{\psi}_k\rangle}{\tilde{\lambda}_k}\tilde{\psi}_k,\quad f\in \mathfrak{H}^s.
\]

We split $\Lambda_q(f)-\Lambda_{\tilde{q}}(f)$, $f\in \mathfrak{H}^s$, into three terms:
\[
\Lambda_q(f)-\Lambda_{\tilde{q}}(f)=A_1+A_2+A_3,
\]
with
\begin{align*}
&A_1= \sum_{k\ge 1}\left(\frac{1}{\tilde{\lambda}_k}-\frac{1}{\lambda_k}\right)\langle f|\tilde{\psi}_k\rangle\tilde{\psi}_k,
\\
&A_2=\sum_{k\ge 1}\frac{\langle f|\tilde{\psi}_k\rangle-\langle f|\psi_k\rangle}{\lambda_k}\tilde{\psi}_k,
\\
&A_3=\sum_{k\ge 1}\frac{\langle f|\psi_k\rangle}{\lambda_k}\left(\tilde{\psi}_k-\psi_k\right).
\end{align*}

Then it is straightforward to check that
\begin{align*}
&\|A_1\|_{L^2(\partial \mathcal{M})}\le C\sum_{k\ge 1}\left|\lambda_k-\tilde{\lambda}_k\right|\|f\| _{L^2(\partial \mathcal{M})},
\\
&\|A_2\|_{L^2(\partial \mathcal{M})}+\|A_3\|_{L^2(\partial \mathcal{M})}\le C\sum_{k\ge 1} \|\tilde{\psi}_k-\psi_k\|_{L^2(\partial \mathcal{M})}\|f\| _{L^2(\partial \mathcal{M})}.
\end{align*}

Therefore, under the assumption $\mathfrak{D}(q,\tilde{q})<\infty$, the operator $\Lambda_q-\Lambda_{\tilde{q}}$ extends to a bounded operator on $L^2(\partial \mathcal{M})$  with
\[
\|\Lambda_q-\Lambda_{\tilde{q}}\|_{\mathscr{B}(L^2(\partial \mathcal{M}))}\le C\mathfrak{D}(q,\tilde{q}).
\]
The proof is then complete.
\end{proof}

\section{The case of incomplete spectral boundary data}

We explain briefly in this short section how we can get a Borg-Levinson type theorem with incomplete spectral boundary data. For this purpose, we assume that $\mathcal{M}$ is simple. In that case it is known that there exists another simple manifold $\mathcal{M}_1$ so that   $\mathcal{M}_1\Supset \mathcal{M}$. 

Fix $\aleph >0$. Let $q, \tilde{q}\in L^\infty (\mathcal{M},\mathbb{R})$ satisfying  $p=(q-\tilde{q})\chi_{\mathcal{M}}\in H^1(\mathcal{M}_1)$,
\[
\|q\|_{L^\infty(\mathcal{M})}\le \aleph ,\quad \|\tilde{q}\|_{L^\infty(\mathcal{M})}\le \aleph\quad  \mbox{and}\quad \|p\|_{H^1(\mathcal{M}_1)}\le \aleph. 
\]
From \cite{BCKS}, we have
\begin{equation}\label{2}
\|p\|_{L^2(\mathcal{M})}^4\le C\left(\frac{1}{|\Im \lambda |}+|\Im\lambda|\|\Lambda_{q-\lambda }- \Lambda_{\tilde{q}-\lambda } \|_{\mathscr{B}(L^2(\partial \mathcal{M} ))}\right),\quad \lambda \in \mathbb{C}\setminus \mathbb{R},
\end{equation}
where the constant $C$ only depends on $n$, $\mathcal{M}$ and $\aleph$.

Now if, for some fixed integer $\ell \ge 1$, $(\lambda_k, \psi_k)=(\tilde{\lambda}_k, \tilde{\psi}_k)$, $k\ge \ell$, then the calculations in the preceding section yield 
\begin{equation}\label{3}
\|\Lambda_{q-\lambda }- \Lambda_{\tilde{q}-\lambda } \|_{\mathscr{B}(L^2(\partial \mathcal{M} ))}\le C\left(\sum_{k=1}^\ell \frac{\lambda_k^2}{|\lambda _k-\lambda|}+\sum_{k=1}^\ell \frac{\tilde{\lambda}_k^2}{|\tilde{\lambda} _k-\lambda|} \right).
\end{equation}
Again, the constant $C$ only depends on $n$, $\mathcal{M}$ and $\aleph$.

We deduce by combining together \eqref{2} and \eqref{3} that there exist two constants $C>0$ and $\kappa >0$, only depending on  $n$, $\mathcal{M}$, $\aleph$ and $\ell$, so that 
\begin{equation}\label{4}
\|p\|_{L^2(\mathcal{M})}^4\le C\left(\frac{1}{|\Im \lambda |}+\frac{|\Im\lambda|}{|\lambda|}\right),\quad \lambda \in \mathbb{C}\setminus \mathbb{R},\; |\lambda|\ge \kappa.
\end{equation}

We find by taking in this inequality $\lambda =(\tau +i)^2=(\tau^2 -1)+2i\tau$, with $\tau \ge 1$ sufficiently large,
\[
\|p\|_{L^2(\mathcal{M})}^4\le C\left(\frac{1}{\tau }+\frac{\tau}{\tau ^2-1}\right).
\]
Taking the limit, as $\tau \rightarrow \infty$, in this inequality, we obtain $p=0$. In other words, we proved that the spectral boundary data $(\lambda_k,\psi_k)_{k\ge \ell}$ determine uniquely $q\in L^\infty (\mathcal{M},\mathbb{R})\cap H_0^1(\mathcal{M})$, for any arbitrary fixed integer $\ell \ge 1$:

\begin{theorem}\label{theorem3}
Let $\ell \ge 1$ be an integer and assume that the manifold $\mathcal{M}$ is simple. If $q,\tilde{q}\in L^\infty(\mathcal{M},\mathbb{R})\cap H_0^1(\mathcal{M})$ satisfy
\[
\lambda_k(q)=\lambda_k(\tilde{q})\quad \mbox{and}\quad \psi_k(q)=\psi_k(\tilde{q}),\quad k\ge \ell,
\]
then $q=\tilde{q}$.
\end{theorem}

\end{document}

\begin{abstract}
Our proof of \cite[Theorem 1.1]{Ch21} is partially incorrect. The reason is that, contrary to what we claimed,  the construction of the space $\mathscr{H}^s$ depends on the orthonormal basis $(\rho_k)$. We give in the present note  a new proof that needs a little bit stronger assumption.
\end{abstract}

\maketitle

Throughout $\mathcal{M}=(\mathcal{M},\mathfrak{g})$ is a smooth compact Riemannian manifold, of dimension $n\ge3$,  with boundary $\partial \mathcal{M}$. 

We consider, where $q\in L^\infty (\mathcal{M},\mathbb{R})$,  the unbounded operator $L=L(q)$ acting on $L^2(\mathcal{M})$ as follows
\[
L=-\Delta _{\mathfrak{g}}+q\quad \mbox{with}\quad D\left(L \right)=H_0^1(\mathcal{M})\cap H^2(\mathcal{M}).
\]
Here $\Delta _{\mathfrak{g}}$ is the usual Laplace-Beltrami operator, which is given in term of local coordinates by
\[
\Delta_\mathfrak{g}=\frac{1}{\sqrt{|\mathfrak{g}|}}\sum_{j,k=1}^n\frac{\partial}{\partial x_j}\left(\sqrt{|\mathfrak{g}|}\mathfrak{g}^{jk}\frac{\partial}{\partial x_k}\, \cdot \right).
\]

We recall that the spectrum of $L$, as self-adjoint operator with compact resolvent,  is reduced to a sequence of eigenvalues $(\lambda_k)=(\lambda_k (q))$ counted according to their multiplicity:
\[
-\infty < \lambda_1\le \ldots \lambda_k\le \ldots \quad \mbox{and}\quad \lambda_k\rightarrow \infty \quad \mbox{as}\quad k\rightarrow \infty.
\]

Furthermore, there exists $(\phi_k)$ an orthonormal basis of $L^2(\mathcal{M})$, $\phi_k=\phi_k(q)$,  consisting of eigenfunctions, with $\phi_k$  associated to $\lambda_k$, $k\ge 1$.

We set, for notational convenience,
\[
\psi_k=\partial_\nu  \phi_k|{_{\partial \mathcal{M}}}, \quad k\ge 1,
\]
where $\nu$ is the unit normal exterior vector field on $\partial \mathcal{M}$ with respect to $\mathfrak{g}$.

Let
\[
L_+^\infty (\mathcal{M})=\{q\in L^\infty (\mathcal{M},\mathbb{R});\; q\ge 0 \; \mbox{(almost everywhere)}\}.
\]

It is clear that if $q\in L_+^\infty (\mathcal{M})$ then $0$ does not belong to the spectrum of $L$ and the eigenvalues of $L$ are all positive. Thus, for any $q\in L_+^\infty (\mathcal{M})$ and  $f\in H^{1/2}(\partial \mathcal{M})$, the BVP
\begin{equation}\label{bvp}
(-\Delta _{\mathfrak{g}}+q)u=0\; \mbox{in}\; \mathcal{M}\quad \mbox{and}\quad  u=f\; \mbox{on}\; \partial \mathcal{M}
\end{equation}
admits a unique solution $u=u(q)(f)\in H^1(\mathcal{M})$ with $\partial_\nu u\in H^{-1/2}(\partial \mathcal{M})$. Moreover, 
\[
\|\partial_\nu u\|_{H^{-1/2}(\partial \mathcal{M})}\le C\|f\|_{H^{1/2}(\partial \mathcal{M})},
\]
where the constant $C$ is independent of $f$. In other words, the mapping
\[
\Lambda(q) :f\in H^{1/2}(\partial \mathcal{M})\mapsto \partial_\nu u\in H^{-1/2}(\partial \mathcal{M})
\]
defines a bounded operator.

The operator $\Lambda(q)$ is  called the Dirichlet-to-Neumann map associated to $q$.

We fix  $0<\epsilon<1/2$  and we let
\[
\alpha_k=k^{-(1-2\epsilon)/(2n)},\quad \beta_k=k^{(1+2\epsilon)/n},\quad  k\ge 1.
\]

We set, where $q,\tilde{q}\in L_+^\infty (\mathcal{M})$,  
\begin{align*}
&\lambda_k=\lambda_k (q),\quad \phi_k=\phi_k (q),\quad \psi_k=\psi_k (q),
\\
&\tilde{\lambda}_k=\lambda_k (\tilde{q}),\quad \tilde{\phi}_k=\phi_k (\tilde{q}),\quad \tilde{\psi}_k=\psi_k (\tilde{q}).
\end{align*}
We define then 
\begin{equation}\label{dist}
\mathfrak{D}(q,\tilde{q})=\sum_{k\ge 1}\alpha_k^2\left|\lambda_k-\tilde{\lambda}_k\right|+\sum_{k\ge 1} \alpha_k\|\tilde{\psi}_k-\psi_k\|_{L^2(\partial \mathcal{M})},
\end{equation}
and
\begin{equation}\label{dist+}
\mathfrak{D}_+(q,\tilde{q})=\sum_{k\ge 1}\beta_k\left|\lambda_k-\tilde{\lambda}_k\right|+\sum_{k\ge 1} \alpha_k\|\tilde{\psi}_k-\psi_k\|_{L^2(\partial \mathcal{M})}.
\end{equation}

We have to  replace \cite[Theorem 1.1]{Ch21} by the following one.

 \begin{theorem}\label{theorem1} 
Fix $\aleph >0$ and let $q,\tilde{q}\in L_+^\infty (\mathcal{M})$ satisfying $q\le \aleph$, $\tilde{q}\le \aleph$  and $\mathfrak{D}_+(q,\tilde{q})<\infty$. Then $\Lambda (q) -\Lambda(\tilde{q})$ extends to a bounded operator on $L^2(\partial \mathcal{M})$ and
\[
\|\Lambda (q) -\Lambda(\tilde{q})\|_{\mathscr{B}(L^2(\partial \mathcal{M}))}\le C\mathfrak{D}(q,\tilde{q}),
\]
where $\mathfrak{D}(q,\tilde{q})$ is given by \eqref{dist}, $\mathfrak{D}(q,\tilde{q})_+$ is as in \eqref{dist+} and $C=C(n,\mathcal{M}, s,\aleph)>0$ is a constant.
 \end{theorem}
 
In light of this new Theorem, \cite[Theorem 1.2]{Ch21} (Borg-Levinson type theorem) still holds. While \cite[Corollary 1.1]{Ch21} remains true if we substitute $\mathfrak{D}(q,\tilde{q})<\infty$ by $\mathfrak{D}_+(q,\tilde{q})<\infty$.
 
In what follows, $C$ will denote a generic constant that may depends only of $n$, $\mathcal{M}$, $s$ and $\aleph$.

Prior to proving Theorem \ref{theorem1} we give a priori estimate of eigenfunctions in the space $H^{3/2+\epsilon}(\mathcal{M})$. Let  $q$ and $\tilde{q}$ be as in Theorem \ref{theorem1}. We state the results only for $\phi_k$ but they still hold also for $\tilde{\phi}_k$. First, from the usual $H^2$ a priori estimate we have
\[
\|\phi\|_{H^2(\mathcal{M})}\le C\lambda_k,\quad k\ge 1,
\]
which, combined with interpolation inequality 
\[
\|w\|_{H^{3/2+\epsilon}(\mathcal{M})}\le C\|w\|_{H^2(\mathcal{M})}^{(3+2\epsilon)/4}\|w\|_{L^2(\mathcal{M})}^{(1-2\epsilon)/4},\quad w\in H^2(\mathcal{M}),
\]
(e.g. \cite[Subsection 2.1 in page 40]{LM}), yields
\begin{equation}\label{e1}
\|\phi_k\|_{H^{3/2+\epsilon}(\mathcal{M})}\le C\lambda_k^{(3+2\epsilon)/4},\quad k\ge 1.
\end{equation}
From Weyl's asymptotic formula and the min-max principle, we have
\begin{equation}\label{wf}
c^{-1}k^{2/n}\le \lambda_k\le ck^{2/n},
\end{equation}
for some constant $c=c(n,\mathcal{M},\aleph)>1$. Whence \eqref{e1} implies
\begin{equation}\label{e2}
\|\phi_k\|_{H^{3/2+\epsilon}(\mathcal{M})}\le Ck^{(3+2\epsilon)/(2n)},\quad k\ge 1.
\end{equation}

In light of the fact that the trace operator 
\[
\tau:w\in H^{3/2+\epsilon}(\mathcal{M})\mapsto \partial_\nu w\in L^2(\partial \mathcal{M})
\]
 is bounded (e.g. \cite[Theorem 9.4 in page 41]{LM}), we get from \eqref{e2} 
\begin{equation}\label{e3}
\|\psi_k\|_{L^2(\partial \mathcal{M})}\le Ck^{(3+2\epsilon)/(2n)},\quad k\ge 1.
\end{equation}

\begin{proof}[Proof of Theorem \ref{theorem1}]
Pick $f\in H^{3/2}(\partial \mathcal{M})$ and define
\[
a_k=-\frac{\langle f|\psi_k\rangle}{\lambda_k}\psi_k,\quad \tilde{a}_k=-\frac{\langle f|\tilde{\psi}_k\rangle}{\tilde{\lambda}_k}\tilde{\psi}_k.
\]
Then set
\[
w_\ell=-\sum_{k= 1}^\ell(a_k- \tilde{a}_k),\quad \ell \ge 1.
\]
We decompose, for each $\ell \ge 1$, $w_\ell$ into three terms $w_\ell=w_\ell^1+w_\ell^2+w_\ell ^3$, with
\begin{align*}
&w_\ell^1= \sum_{k= 1}^\ell \left(\frac{1}{\tilde{\lambda}_k}-\frac{1}{\lambda_k}\right)\langle f|\tilde{\psi}_k\rangle\tilde{\phi}_k,
\\
&w_\ell^2=\sum_{k= 1}^\ell \frac{\langle f|\tilde{\psi}_k\rangle-\langle f|\psi_k\rangle}{\lambda_k}\tilde{\phi}_k,
\\
&w_\ell^3=\sum_{k= 1}^\ell \frac{\langle f|\psi_k\rangle}{\lambda_k}\left(\tilde{\phi}_k-\phi_k\right).
\end{align*}

We obtain in a straightforward manner from \eqref{wf} and \eqref{e3}
\begin{align*}
&\|\tau w_\ell^1\|_{L^2(\partial \mathcal{M})}\le C\sum_{k=1}^\ell \alpha_k^2|\lambda_k-\tilde{\lambda}_k|\|f\|_{L^2(\partial \mathcal{M})}\le C\mathfrak{D}(q,\tilde{q})\|f\|_{L^2(\partial \mathcal{M})} ,
\\
&\|\tau w_\ell^2\|_{L^2(\partial \mathcal{M})}\le C\sum_{k= 1}^\ell \alpha_k \|\psi_k-\tilde{\psi}\|_{L^2(\partial \mathcal{M})}\|f\|_{L^2(\partial \mathcal{M})}\le C\mathfrak{D}(q,\tilde{q})\|f\|_{L^2(\partial \mathcal{M})},
\\
&\|\tau w_\ell^3\|_{L^2(\partial \mathcal{M})} \le C\sum_{k= 1}^\ell \alpha_k \|\psi_k-\tilde{\psi}\|_{L^2(\partial \mathcal{M})}\|f\|_{L^2(\partial \mathcal{M})}\le C\mathfrak{D}(q,\tilde{q})\|f\|_{L^2(\partial \mathcal{M})}.
\end{align*}

We used above that $\mathfrak{D}(q,\tilde{q})\le \mathfrak{D}_+(q,\tilde{q})$.

In consequence, the series  $\sum_{k\ge 1}(a_k-\tilde{a}_k)$ converges  in $L^2(\partial \mathcal{M})$. Furthermore, we have
\begin{equation}\label{e4}
\left\| \sum_{k\ge 1}(a_k-\tilde{a}_k) \right\|_{L^2(\partial \mathcal{M})}\le C\mathfrak{D}(q,\tilde{q})\|f\|_{L^2(\partial \mathcal{M})}.
\end{equation}

We consider the following temporary notations, where $\lambda \in \mathbb{R}\setminus\{0\}$,
\begin{align*}
u=u(q)(f),\quad \tilde{u}=u(\tilde{q})(f),\quad u(\lambda )=u(q+|\lambda])&(f)
\\
&\mbox{and}\quad \tilde{u}(\lambda )=u(\tilde{q}+|\lambda])(f).
\end{align*}

Similarly to \cite[Lemma 2.2]{So}, we find
\begin{align*}
&\tau u-\tau u(\lambda )=-\sum_{k\ge 1} \frac{|\lambda|\langle f|\psi_k\rangle}{\lambda_k(|\lambda|+\lambda_k)}\psi_k=\sum_{k\ge 1}\frac{|\lambda| }{|\lambda|+\lambda_k}a_k,
\\
&\tau \tilde{u}-\tau \tilde{u}(\lambda )=-\sum_{k\ge 1} \frac{|\lambda|\langle f|\tilde{\psi}_k\rangle}{\tilde{\lambda}_k(|\lambda|+\tilde{\lambda}_k)}\tilde{\psi}_k=\sum_{k\ge 1}\frac{|\lambda|}{|\lambda|+\tilde{\lambda}_k}\tilde{a}_k,
\end{align*}
and the convergence holds in $L^2(\partial \mathcal{M})$.

We derive from these two identities the following decomposition
\begin{align*}
\tau u-\tau \tilde{u}-\sum_{k\ge 1}(a_k-\tilde{a}_k)&=\tau u(\lambda )-\tau \tilde{u}(\lambda )
\\
&+\sum_{k\ge 1}\left[\left(\frac{|\lambda| }{|\lambda|+\lambda_k}-1\right)a_k-\left(\frac{|\lambda|}{|\lambda|+\tilde{\lambda}_k}-1\right)\tilde{a}_k\right].
\end{align*}

We proceed as in the proof of \cite[Lemma 2.3]{So} in order to obtain that
\[
\lim_{\lambda \rightarrow \infty}\| \tau u(\lambda )-\tau \tilde{u}(\lambda )\|_{L^2(\partial \mathcal{M})}=0.
\]
We will then get 
\[
\Lambda (q)(f)-\Lambda(\tilde{q})(f)=\tau u-\tau \tilde{u}=\sum_{k\ge 1}(a_k-\tilde{a}_k),
\]
if we prove that
\begin{equation}\label{e5}
\lim_{\lambda \rightarrow \infty}\left\| \sum_{k\ge 1}\left[\left(\frac{|\lambda| }{|\lambda|+\lambda_k}-1\right)a_k-\left(\frac{|\lambda|}{|\lambda|+\tilde{\lambda}_k}-1\right)\tilde{a}_k\right]\right\|_{L^2(\partial \mathcal{M})}=0.
\end{equation}

Since 
\[
\left\| \left(\frac{|\lambda| }{|\lambda|+\lambda_k}-1\right)(a_k-\tilde{a}_k)\right\|_{L^2(\partial \mathcal{M})}\le 2 \| a_k-\tilde{a}_k\|_{L^2(\partial \mathcal{M})},
\]
we can apply dominated convergence theorem to conclude that 
\begin{equation}\label{e6}
\lim_{\lambda \rightarrow \infty}\left\|\sum_{k\ge 1}\left(\frac{|\lambda| }{|\lambda|+\lambda_k}-1\right)(a_k-\tilde{a}_k)\right\|_{L^2(\partial \mathcal{M})}=0.
\end{equation}
On the other hand, since 
\[
\|\tilde{a}_k\|_{L^2(\partial \mathcal{M})}\le C\beta_k\|f\|_{L^2(\partial \mathcal{M})},
\] 
we obtain in a straightforward manner  
\[
\left\|\left(\frac{|\lambda| }{|\lambda|+\lambda_k}-\frac{|\lambda|}{|\lambda|+\tilde{\lambda}_k}\right)\tilde{a}_k\right\|_{L^2(\partial \mathcal{M})}\le C|\lambda |^{-1}\|f\|_{L^2(\partial \mathcal{M})}\beta_k|\lambda_k-\tilde{\lambda}_k|.
\]
Taking into account that $\mathfrak{D}_+(q,\tilde{q})<\infty$, we find 
\begin{equation}\label{e7}
\lim_{\lambda \rightarrow \infty}\left\|\sum_{k\ge 1}\left(\frac{|\lambda| }{|\lambda|+\lambda_k}-\frac{|\lambda|}{|\lambda|+\tilde{\lambda}_k}\right)\tilde{a}_k\right\|_{L^2(\partial \mathcal{M})}=0.
\end{equation}

We deduce that \eqref{e5} holds by putting together \eqref{e6} and \eqref{e7}. We then complete the proof by observing that we have from \eqref{e4}
\[
\| \Lambda (q)(f)-\Lambda(\tilde{q})(f)\|_{L^2(\partial \mathcal{M})}=\left\|\sum_{k\ge 1}(a_k-\tilde{a}_k)\right\|_{L^2(\partial \mathcal{M})}\le C\mathfrak{D}(q,\tilde{q})\|f\|_{L^2(\partial \mathcal{M})}.
\]
Therefore $\Lambda(q)-\Lambda(\tilde{q})$ extends to a bounded operator on $L^2(\partial \mathcal{M})$ and
\[
\|\Lambda(q)-\Lambda(\tilde{q})\|_{\mathscr{B}(L^2(\partial \mathcal{M}))}\le C\mathfrak{D}(q,\tilde{q}),
\]
as expected.
\end{proof}

\subsection*{Acknowledgement} I would like to thank El Maati Ouhabaz for his valuable comments.